\documentclass[1p]{elsarticle}

\usepackage{lineno,hyperref}

\journal{Journal of \LaTeX\ Templates}

\usepackage{amssymb}
\usepackage{eucal}
\newcommand{\R}{{\Bbb R}}

\newcommand{\N}{{\Bbb N}}
\newcommand{\C}{{\Bbb C}}

\usepackage{amsmath}

\newtheorem{thm}{Theorem}
\newtheorem{lemma}[thm]{Lemma}
\newtheorem{corollary}[thm]{Corollary}

\newproof{proof}{Proof}

\begin{document}

\begin{frontmatter}


\title{A simple approach to the wave uniqueness problem}

\author[a]{Abraham Solar}
\author[b]{and Sergei Trofimchuk}
\address[a]{Instituto de F\'isica, Pontificia Universidad Cat\'olica de Chile, Casilla 306, Santiago, Chile 
\\ {\rm E-mail: asolar@fis.uc.cl}}
\address[b]{Instituto de Matem\'atica y F\'isica, Universidad de Talca, Casilla 747,
Talca, Chile \\ {\rm E-mail: trofimch@inst-mat.utalca.cl}}

\begin{abstract}  
We propose a new approach for proving uniqueness of  semi-wavefronts in generally non-monotone  monostable  reaction-diffusion equations with distributed delay.  This allows to solve an open problem concerning the uniqueness of non-monotone (hence, slowly oscillating) semi-wavefronts to the KPP-Fisher equation with delay. Similarly, a broad family of the Mackey-Glass type diffusive equations is shown to possess a unique (up to translation)  semi-wavefront 
for each admissible speed. 
\end{abstract}


\end{frontmatter}


\section{Introduction and main results}
The uniqueness  of traveling waves for the monostable delayed or non-local reaction-diffusion equations is an important and `largely open' \cite{chen}  question of the theory of  partial functional differential equations.  In consequence, different strategies have  been elaborated so far to tackle the uniqueness problem, e.g. see  
\cite{AGT,BS,BNPR,CC,chen,co,CDM,dk,fzJDE,FZ,GT,HTa,mazou,MeiI,tz,Sch,TPT,WLR,XX}.   Broadly speaking, the cited works show that the wave uniqueness can be established when either the  evolution equation or  the waves are monotone,  or when the Lipschitz constant of nonlinear reaction term is dominated by its derivative  at the unstable  equilibrium (the Diekmann-Kaper  condition).  On the other hand, recent studies \cite{CDM,HK,NPTT} reveal that the uniqueness property  can fail to hold even for monotone waves of some non-local monostable equations.  To get a clearer picture of the situation, consider the following KPP-Fisher delayed equation 
(see \cite{BS, ADN,GT,LMS,HK,HTa,HTw,KO,wz} for more detail and  references concerning this model): 
\begin{equation}\label{KPP} \hspace{0mm}
\frac{\partial u(t,x)}{\partial t} = \frac{\partial^2 u(t,x)}{\partial x^2}  + u(t,x)(1-u(t-h,x)). 
\end{equation}
This partial functional differential equation does not meet quasi-monotonicity assumptions neither in the sense of Wu-Zou \cite{wz} nor in the sense of Martin-Smith \cite{MS}. Moreover, its reaction term  does not satisfy the Diekmann-Kaper  dominance condition at $0$.  In addition,  if $h \geq 0.57$ then all non-constant positive wave solutions to (\ref{KPP}) are slowly oscillating in the space variable, see \cite{ADN,HTa}. In this situation,  neither the comparison techniques nor the  Berestycki-Nirenberg sliding solutions method, nor the  Diekmann-Kaper approach can be used to prove the uniqueness of all traveling waves to (\ref{KPP}) (the property conjectured in \cite{HK}).  Thus only partial waves' uniqueness results for the Hutchinson diffusive equation  (\ref{KPP}) were available so far. For instance,   
for some speeds $c\geq 2$ and delays $h\leq 0.57$, this equation has monotone wavefronts $u(t,x) = \phi(x+ct)$,  their uniqueness (up to translation)
was proved in \cite{GT,HTa}. Noteworthily, in the recent e-print \cite{BS}, this result was complemented and the uniqueness of all  fast waves, $c \geq 2\sqrt{2}$,  for (\ref{KPP}) (including non-monotone waves) was deduced from their global stability on semi-infinite intervals.

A similar situation is also observed for another popular delayed  model, the Mackey-Glass type  diffusive equation \cite{AGT,BY,fzJDE,mazou,MeiI,tz,TPT,TPTJDE,WLR,XX,yz}
\begin{equation}\label{MG} \hspace{0mm}
\frac{\partial u(t,x)}{\partial t} = \frac{\partial^2 u(t,x)}{\partial x^2}   -u(t,x) + g(u(t-h,x)), \ u \geq 0. 
\end{equation}
Here $g:\R_+\to \R_+$ is unimodal (i.e. hump-shaped) bounded $C^1$-smooth function possessing exactly two non-negative fixed points $u_1\equiv 0<u_2\equiv \kappa$. For (\ref{MG}),   the uniqueness question is completely (i.e. for all $h\geq 0$ and for all admissible  speeds)  answered only if either $g$ is monotone on $[0,\kappa]$,  or $|g'(s)| \leq g'(0), \ s \in \R,$ (this amounts to the  Diekmann-Kaper condition at the equilibrium $0$ for (\ref{MG})).  This is  for instance the case of the Nicholson's diffusive equation ($g(u)=pue^{-u}$); however, other population  models (like R. May's sei whale model \cite{BY}, where $g(u) = \max\{pu(1-u^z/k^z),0\}$, for some $z >1$)  do not fit into the frameworks of the above mentioned theories. 
 
In the present work,  we propose a novel approach for proving uniqueness of  semi-wavefronts in a general non-monotone monostable  reaction-diffusion equations with distributed delay of the following form 
\begin{equation}\label{17a} \hspace{0mm}
\frac{\partial u(t,x)}{\partial t} = \frac{\partial^2 u(t,x)}{\partial x^2}  + f(u_t(\cdot,x)), \ u
\geq 0, 
\end{equation}
where $f: C[-h,0] \to \R$ is a continuous functional, $C:= C[-h,0]$ is the Banach space of all continuous scalar functions defined on $[-h,0]$ and 
$u_t(s,x) = u(t+s,x),\ s \in [-h,0]$, belongs to $C$ for every fixed $x\in \R$. 
 This approach is based on a relatively simple idea which  we believe can also be  useful for other diffusive systems.  In particular, our 
method allows to provide a complete solution to the uniqueness problem for the KPP-Fisher equation with delay or for the May diffusive  baleen whale model. 

\vspace{0mm}

In the sequel,   we always assume the following  natural and easily verifiable conditions: 

\vspace{1mm}

\noindent {\bf (M)} Identifying each constant function $x\in C$ with a real number $x \in \R$, 
set $f^*(x) = f(x)$ and suppose that $f^*(x)$ satisfies the standard monostability requirements
(i.e. $f^*$ has only two zeros on $\R_+$, $x_1=0$ and $x_2=\kappa$; moreover, $f^*(x) >0$ on $(0,\kappa)$). 
 
\vspace{1mm}

\noindent {\bf (S)} Functional $f:C \to \R$ is continuous, transforms bounded sets into bounded sets,  and it is differentiable at $0$. Moreover, for some positive $\alpha$, $\delta$, $K$, and the max-norm $|\phi|_C$, 
\begin{equation}\label{S}
|f(\psi) -f(\phi) - f'(0)(\psi-\phi)| \leq K|\psi-\phi|_C(|\phi|_C^{\alpha}+|\psi|_C^{\alpha}),  \ \mbox{for all} \ |\phi|_C<\delta,  |\psi|_C<\delta. 
\end{equation}

\noindent Using the Jordan decomposition theorem, we can write $f'(0)\phi$ as 
$$
f'(0)\phi = \int_{-h}^0\phi(s)d\mu_+(s) - \int_{-h}^0\phi(s)d\mu_-(s), 
$$
where $\mu_\pm$ are non-decreasing functions on $[-h,0]$. 

\vspace{1mm}

\noindent {\bf (J)}  We will assume that $\int_{-h}^0\phi(s)d\mu_-(s) = q \phi(0)$ for some $q\geq 0$.

\vspace{1mm}

\noindent By {\bf (J)}, we obtain that   
$$f'(0)\phi = -q\phi(0)+  \int_{-h}^0\phi(s)d\mu_+(s). $$

\noindent {\bf (ND)}   Set $p:=  \int_{-h}^0d\mu_+(s)$, it follows from {\bf (M)}  that $p \geq q$. In addition, we will assume the 
following {non-degeneracy}  condition: $p >q$.

\vspace{1mm}

\noindent {\bf (UB)} For each $\phi, \psi \in C$ satisfying $0 < \phi(s)\leq \psi (s), \ s \in [-h,0],$ it holds
that \footnote{
It is instructive to note that, by its essence,  {\it `quasi-monotonicity'} condition {\bf (UB)} is completely  different from the Wu-Zou quasi-monotonicity condition (A2) introduced in  \cite{wz}.}
$$f(\psi) - f(\phi) \leq  f'(0)(\psi-\phi).$$ 
\noindent {\bf (LB)}  Moreover,  for each $\epsilon >0$ there exists $\delta >0$ such that for all $\phi \in C$ satisfying $0<\phi(s)\leq \delta,$ $s \in [-h,0]$, it holds that 
$$q\phi(0)+ f(\phi) \geq (1-\epsilon) \int_{-h}^0\phi(s)d\mu_+(s). $$

\vspace{1mm}

\noindent Our main result is the following theorem.  

\begin{thm} \label{main} Let $u_1(x,t) = \phi(x+ct)$, $u_2(x,t) = \psi(x+ct)$ be two positive semi-wavefronts (i.e. $\phi(-\infty)=\psi(-\infty)=0$ and 
$\phi(t), \psi(t)$ are positive and bounded on $\R$) of  equation (\ref{17a}). If assumptions  {\bf (M)},  {\bf (S)}, {\bf (J)},  {\bf (ND)}, {\bf (LB)} and  {\bf (UB)} are satisfied then 
$\phi(t+t')\equiv \psi(t),$ $t \in \R$, for some finite $t'\in \R$.  
\end{thm}

\begin{corollary} For each $c\geq 2$ and for each $h\geq 0$,  the KPP-Fisher delayed equation (\ref{KPP}) has a unique (up to translation) semi-wavefront.  
\end{corollary}
\begin{proof} By \cite{HTa}, for each $c\geq 2$ and $h \geq 0$, equation (\ref{KPP}) has at least one semi-wavefront. For  $\phi  \in C$, set 
$f(\phi)= \phi(0)(1-\phi(-h))$, then $f'(0)\phi=\phi(0)$, $ p =1>q =0$, and
$$
f(\psi) - f(\phi)  = \psi(0)-\phi(0) +\phi(0)\phi(-h)-\psi(0)\psi(-h) =  f'(0)(\psi-\phi) +
$$
$$
\phi(0)(\phi(-h)-\psi(-h)) + \psi(-h)(\phi(0)-\psi(0))\leq  f'(0)(\psi-\phi) \ \mbox{once} \ 0< \phi(s) \leq  \psi(s), s \in [-h,0];
$$
$$
|f(\psi) -f(\phi)-f'(0)(\psi-\phi)|  \leq |\psi-\phi|_C(|\phi|_C+ |\psi|_C)  \ \mbox{for all } \   \phi,\psi  \in C.  
$$
Thus assumptions  {\bf (M)},  {\bf (S)}, {\bf (J)}, {\bf (UB)}  and  {\bf (LB)}, {\bf (ND)} are clearly satisfied.  Therefore the statement of the corollary follows from Theorem \ref{main}. \hfill $\square$
\end{proof}
To present a similar result for the Mackey-Glass type diffusive equation (\ref{MG}), we need the following auxiliary 
assertion. 
\begin{lemma} \label{STA} Assume {\bf (J)},  {\bf (ND)}   and set $\chi(z,c) = z^2 -cz + f'(0)e^{\cdot cz}$. Then there exists $c_*>0$ such that $\chi(z,c)$ has exactly two positive zeros (counting multiplicity) $\lambda_1(c)\leq \lambda_2(c)$ if and only if $c \geq c_*$. These zeros are simple  if $c > c_*$, while  $\lambda_1(c_*) = \lambda_2(c_*)$ is a double zero.  Next, for $c \geq c_*$ every different zero $\lambda_j(c)$ of  $\chi(z,c)$  satisfies $\Re \lambda_j(c) < \lambda_1(c)$. 
\end{lemma}

The proof of Lemma \ref{STA} uses standard arguments of the complex analysis, for the convenience of the reader it is given in the appendix. 

\vspace{1mm}

Hence, we have the following existence and uniqueness result for equation (\ref{MG}):
\begin{corollary} \label{MGco} Suppose that the real Lipschitz continuous function $-x+g(x)$ satisfies the monostability and smoothness conditions of {\bf (M)},  {\bf (S)} where the space $C$ 
is replaced with $\R$. If, furthermore, $g'(0) >1$ and  $g(x_2)-g(x_1) \leq g'(0)(x_2-x_1)$ for all $x_1 <x_2$, then 
for each $c\geq c_*$ and for each $h\geq 0$,  the diffusive delayed equation (\ref{MG}) has a unique (up to translation) semi-wavefront.
\end{corollary}
\begin{proof}  For  (\ref{MG}),   the semi-wavefront existence (for all $c\geq c_*$ and for all $h\geq 0$) was proved in  \cite[Theorem 18]{gpt}. With $f(\phi) = -\phi(0) + g(\phi(-h))$, $f'(0)\phi = -\phi(0) + g'(0)\phi(-h)$,  $q=1 < p =g'(0)$,  verification of other assumptions of Theorem \ref{main} is 
an easy task.   \hfill $\square$
\end{proof}
Corollary \ref{MGco} does apply to the above mentioned May diffusive  baleen whale model.\footnote{In is easy to show that in this model each 
semi-wavefront $\phi$ satisfies the inequality $\phi(t) < k, \ t \in \R$.}
\section{Proof of Theorem \ref{main}} \label{S1}
\noindent The proof of Theorem \ref{main} is divided into the following four parts.
\subsection{Proof of the exponential decay of wave profiles  at  $-\infty$.}  By the definition of a semi-wavefront $u(t,x)=\phi(x+ct)$, it holds that  $\phi(-\infty)=0$.  To prove the uniqueness of $\phi$, we need to derive more detailed information concerning asymptotic behavior of $\phi$ at $-\infty$. In this subsection, under assumptions imposed in  the introduction, we establish  that $\phi(t)$ decays exponentially at $-\infty$.  This  property is well known from  \cite{HVL} in the case when  $f$ is $C^1$-smooth and bounded, together with its Fr\'echet derivative, in some vicinity of $0\in C$ and when, in addition, $\chi(z,c)$ does not have zeros on the imaginary axis (clearly, this is true for the KPP-Fisher delayed equation). 
However, for some admissible pairs $(c,f'(0))$, function $\chi(z,c)$ can have purely imaginary zeros, and therefore we should prove  
exponential decay of $\phi$ at $-\infty$ even if $0$ is non-hyperbolic equilibrium of the profile equation
\begin{equation}\label{pf}
\phi''(t) -c\phi'(t) + f (\tilde \phi_t) =0,  \ \phi(-\infty)=0, \ \phi(t) >0, \  \sup_{t \in \R}|\phi(t)| <\infty.
\end{equation}
Here $\tilde \phi_t\in C$ is defined by $\tilde \phi_t(s) = \phi (t+cs), \ s \in [-h,0]$.  Observe that  the analysis  of the rate of  decay of wave profiles at $-\infty$ is an important  part of  proofs of almost all wave uniqueness theorems (e.g. cf. \cite{AGT,CC,dk,mazou, WLR,XX}).  

\begin{lemma} \label{EDE} Assume {\bf (J)},  {\bf (ND)} and {\bf (LB)}, {\bf (UB)}. Then for each semi-wavefront profile $\phi$ there exists $\gamma >0$ such that
$\phi(t)+ |\phi'(t)| = O(e^{\gamma t})$ as $t \to -\infty$. 
\end{lemma}
\begin{proof} Since the wave profile $\phi$ is a  bounded function, it satisfies the integral equation 
\begin{equation}\label{bae}
\phi(t) = \int_{-\infty}^{+\infty} K(t-s)[(1+q)\phi(s)+f(\tilde \phi_s)]ds,\quad t \in \R, 
\end{equation}
where $K$ is the positive Green function (the fundamental solution, cf. \cite{TPTJDE}) of the equation 
$
y''(t)-cy'(t) - (1+q)y(t) =0. 
$
Take $\epsilon >0 $ in  {\bf (LB)} so small that $(1-\epsilon)p > q$ and let $s'$ be such that $\phi(s) < \delta = \delta(\epsilon)$ for all $s \leq s'$.  Consider
$$G(s): = (1-\epsilon)  \int_{-h}^0K(s+c\sigma)d\mu_+(\sigma)+ K(s),  \quad \int_{-\infty}^{+\infty}G(s)ds = \frac{1+(1-\epsilon)p}{1+q} >1,$$
and take $N>ch$ large enough to satisfy $\int_{-N}^NG(s)ds >1$. 
In view of {\bf (LB)}, for each  $t < s'-N-ch$, it holds that 
$$
\phi(t) \geq  \int_{t-N}^{t+N+ch} K(t-s)[\phi(s)+(1-\epsilon) \int_{-h}^0\phi(s+c\sigma)d\mu_+(\sigma)]ds \geq \int_{t-N}^{t+N}G(t-s)\phi(s)ds. 
$$
Thus for $t' < t < s'-2N$, we obtain that 
\begin{equation}\label{ine1}
\int_{t'}^t\phi(v)dv \geq \int_{-N}^NG(s)\int_{t'}^t\phi(v-s)dvds, \quad \mbox{where} \ \int_{-N}^NG(s)ds >1. 
\end{equation}
As it was shown in \cite[Theorem 1]{AGT}, inequality (\ref{ine1}) implies that $\int_{-\infty}^0e^{-\gamma s}\phi(s)ds$ converges for some positive $\gamma $. 

Next, due to {\bf (UB)}, we obtain from (\ref{bae}) that 
\begin{equation}\label{bae}
\phi(t) \leq  \int_{-\infty}^{+\infty} K(t-s)[(1+q)\phi(s)+f'(0)\tilde \phi_s]ds =  \int_{-\infty}^{+\infty} G_1(t-s)\phi(s)ds,\quad t \in \R, 
\end{equation}
where 
$$G_1(s): =  \int_{-h}^0K(s+c\sigma)d\mu_+(\sigma)+ K(s), \ s \in \R.$$
Since the bounded function $K$ satisfies $K(s) = O(e^{ct}), \ t \to -\infty$, we have that $G_1(s) \leq Ae^{cs},$ $s \in \R,$ for some positive $A$. In consequence, for $\gamma  \in (0,c)$ as above, we obtain that $G_1(s) \leq Be^{\gamma  s},$ $ s \in \R,$ for some $B>0$. Thus $\phi(t) = O(e^{\gamma t})$ as $t \to -\infty$ because of 
$$
\phi(t)e^{-\gamma  t} \leq  \int_{-\infty}^{+\infty} G_1(t-s)e^{-\gamma (t-s)}\phi(s)e^{-\gamma  s}ds \leq B\int_{-\infty}^{+\infty}e^{-\gamma s}\phi(s)ds=:D <\infty ,\quad t \in \R. 
$$
Similarly, solving (\ref{pf}) with respect to $\phi'(t)$, we find that 
\begin{equation}\label{pro}
\phi'(t) = \int_t^{+\infty}e^{c(t-s)}f(\tilde \phi_s)ds, \ t \in \R.
\end{equation}
Next, {\bf (LB)}, {\bf (UB)} and the exponential estimate for $\phi(t)$ implies that, for some $T_1 \in \R$,  $D_1 >0$,  
$$
|f(\tilde \phi_s)| \leq q\phi(s)+\int_{-h}^0\phi(s+cs)d\mu_+(s)\leq D_1e^{\gamma s}, \quad s \leq T_1. 
$$
Since $|f(\tilde \phi_s)|$ is a bounded function on $\R$ (cf. {\bf (S)}),  for some $D_2>0$, we conclude that $|f(\tilde \phi_s)| \leq  D_2e^{\gamma s}, \ s \in \R$. 
 Then (\ref{pro}) implies the following: 
$$
|\phi'(t)|e^{-\gamma t} \leq \int_t^{+\infty}e^{c(t-s)}D_2e^{-\gamma (t-s)} ds = \frac{D_2}{c-\gamma}, \ t \in \R.
$$
This completes the proof of Lemma \ref{EDE}. \hfill $\square$
\end{proof}

\subsection{Non-existence of super-exponentially decaying solutions at $-\infty$. } 
We will also need the following nonlinear version of Lemma 3.6 in \cite{VT}. It
excludes the existence of small solutions to asymptotically autonomous delayed differential equations at $-\infty$: 
\begin{lemma} \label{sms} Suppose that $L:C([-h,0],\R^n) \to \R^n$ is continuous linear operator 
and $M: (-\infty, 0] \times C([-h,0],\R^n) \to \R^n$ is a continuous function such that 
$|M(t,\phi)| \leq \mu(t)|\phi|_C$ for some non-negative $\mu(t) \to 0$ as $t\to -\infty$.
Then the system 
\begin{equation} \label{LM}
x'(t) = Lx_t+M(t,x_t), \  x_t(s):= x(t+s), \ s \in [-h,0],
\end{equation}
does not have nontrivial exponentially small solutions at $-\infty$ (i.e. non-zero solutions $x:\R_- \to \R^n$ such that for each $\gamma \in \R$ it holds that $x(t)e^{\gamma t} \to 0, \ t \to -\infty$). 
\end{lemma}
\begin{proof}  The proof of Lema \ref{sms} is a slight modification of the proof of Lemma 3.6 in e-print \cite{VT}:   
for the reader's convenience,  it is included in this paper.  So, on the contrary, suppose that there exists a small solution $x(t)$ of $(\ref{LM})$ at $-\infty$. Take some $b>h$.  It is straightforward  to see that the property 
 $x(t)e^{\gamma t} \to 0,$  $t \to -\infty,$ is equivalent to  $|x_t|_be^{\gamma t} \to 0, \ t \to -\infty$, where $|x_t|_b = \max_{s \in [-b,0]}|x(t+s)|$. We claim that smallness of 
 $x(t)$ implies that 
 $
 \inf_{t \leq 0} |x_{t-b}|_b/|x_t|_b =0. 
 $
 Indeed, otherwise there is $K>0$ such that $|x_{t-b}|_b/|x_t|_b \geq K, \ t \leq 0,$ and therefore, setting $\nu:= b^{-1}\ln K$, we obtain the following contradiction: 
 $$
0<  |x_t|_be^{\nu t} \leq |x_{t-b}|_be^{\nu (t-b)}\leq |x_{t-2b}|_be^{\nu (t-2b)}\leq\dots\leq |x_{t-mb}|_be^{\nu (t-mb)} \to 0, \quad m \to +\infty.  
$$
Hence, for $b=3h$  there is  a sequence $t_j\to -\infty$ such that  $|x_{t_j-3h}|_{3h}/|x_{t_j}|_{3h} \to 0$ as $j \to \infty$. Clearly, 
$|x_{t_j}|_{3h} = |x(s_j)|$ for some $s_j \in [t_j-3h,t_j]$ and, for all large $j$, it holds $|x(s_j)| \geq |x(s)|, \ s \in [t_j-6h,t_j]$.   Since $0\leq t_j-s_j \leq 3h$, 
without loss of generality we can assume that $\theta_j:= t_j-s_j \to \theta_* \in [0,3h]$. 

Now, for sufficiently large $j$, consider the sequence 
of functions $$
y_j(t)= \frac{x(t+t_j)}{|x(s_j)|}, \ t \in [-6h,0], \quad |y_j(-\theta_j)| =1, \quad  |y_j(t)| \leq 1, \ t \in [-6h,0]. 
$$
For each $j$, $
y_j(t)$ satisfies the equations 
$$
y'(t) = Ly_t+\frac{M(t+t_j,x_{t+t_j})}{|x(s_j)|}, \quad y_j(t) = y_j(-\theta_j)+ \int^t_{-\theta_j}\left(Ly_u+\frac{M(u+t_j,x_{u+t_j})}{|x(s_j)|}\right)du, 
$$
and therefore   
$
 |y_j(t)| \leq 1, \ |y'(t)| \leq \|L\|+\sup_{s \leq t_j} \mu(s) \leq \|L\|+\sup_{s \leq 0} \mu(s),$ $t\in  [-5h,0],$ $j \in \N
$ (here $\|\cdot\|$ denotes the operator norm).
Thus, due to the Arzel\`a-Ascoli theorem,  there exists a subsequence $y_{j_k}(t)$ converging, uniformly on  $[-5h,0]$, to some continuous function $y_*(t)$ such that 
$|y_*(-\theta_*)|=1, $
$$
y_*(t) = y_*(-\theta_*)+ \int^t_{-\theta_*}L(y_*)_udu,\quad  t\in  [-4h,0]. 
$$
In particular, $y'_*(t) = L(y_*)_t,\  t\in  [-4h,0]$.  Since $y_*(t)= 0$ for all $t\in [-5h,-3h]$,  the existence and uniqueness theorem applied to the initial 
value problem $y'(t) = Ly_t,$ $ t\in  [-3h,0], \ y_{-3h} =0,$ implies that  also $y_*(t)= 0$  for all $t \in [-3h,0]$. However, this contradicts that
$|y_*(-\theta)|=1$. The proof of Lemma \ref{sms} is completed. 
\hfill $\square$
\end{proof}

\subsection{Asymptotic representations of semi-wavefronts   at  $-\infty$.}
The estimate obtained in  Lemma \ref{EDE} can be considerably improved: 
\begin{lemma} \label{EDE2} Assume {\bf (J)},  {\bf (ND)}, {\bf (S)}  and {\bf (LB)}, {\bf (UB)}. Then there exists some $\epsilon >0$ and $t_0 \in \R$ such that 
\vspace{1mm}
$$
(i)\ \ \mbox{if} \ c >c_* \ \mbox{then}\quad   (\phi(t+t_0), \phi'(t+t_0)) = (1,\lambda_1(c)) e^{\lambda_1(c)t} + O(e^{(\lambda_1(c)+\epsilon)t}), \ t \to -\infty; 
$$ 
$$
\hspace{-5mm}(ii)\ \ \mbox{if} \ c =c_* \ \mbox{then}\quad   (\phi(t+t_0), \phi'(t+t_0)) = -(1,\lambda_1(c)) e^{\lambda_1(c)t}(t + O(1)), \ t \to -\infty. 
$$  
\end{lemma}
\begin{proof} Clearly, $\phi(t)$ satisfies the linear inhomogeneous equation  
\begin{equation}\label{pfas}
\phi''(t) -c\phi'(t) -q\phi(t)+ \int_{-h}^0\phi(t+cs)d\mu_+(s) =  Q(t), 
\end{equation}
where, due to assumptions  {\bf (S)} and  {\bf (UB)},  for some $T_2\in \R$ and $C>0$,
$$
Q(t):= f'(0)\tilde \phi_t - f (\tilde \phi_t) \geq 0,\  t \in \R;  \quad |Q(t)| \leq C |\phi_t|^{1+\alpha}_C , \ t \leq T_2. 
$$
Then, in view of the positivity of $\phi(t)$, Lemma 5  together with \cite[Lemma 28]{GT}\footnote{Lemma 28 in \cite{GT} was proved for the case of a single discrete delay, however its proof is valid without significant modifications for the case of delays distributed on a fixed finite interval.} imply that, for sufficiently small
$\varepsilon >0$, it holds that
$$
\phi(t)= - {\rm Res}_{z=\lambda_1(c)} \left [\frac{e^{zt}}{\chi(z,c)} \int_{-\infty}^{+\infty}e^{-zs}Q(s)ds\right] 
+ O(e^{(\lambda_1(c)+\varepsilon)t}), \quad t \to -\infty. 
$$
A straightforward calculation of the above residue (cf. \cite{AGT, GT}) implies the asymptotic formulas for $\phi(t)$ in both cases, $c=c_*$ and $c > c_*$,  whenever  
\begin{equation}\label{pi}
\int_{-\infty}^{+\infty}e^{-\lambda_1(c)s}Q(s)ds >0. 
\end{equation}
Now,  suppose that (\ref{pi}) does not hold. Then $Q(t) \equiv 0$ on $\R$ and, consequently, $\phi(t)$ solves the homogeneous equation 
$$
\phi''(t) -c\phi'(t) -q\phi(t)+ \int_{-h}^0\phi(t+cs)d\mu_+(s) =  0, \ t \in \R.
$$
By Lemma \ref{sms}, this equation does not have nontrivial small solutions at $-\infty$.  But then \cite[Lemma 28]{GT} and Lemma \ref{STA} assure
that $\phi(t)$ is a linear combination of the eigenfunctions $e^{\lambda_1(c) t}$,  $e^{\lambda_2(c) t}$.  This means that $\phi(t)$ is unbounded on $\R$, a contradiction proving inequality (\ref{pi}). 

Finally, assuming $c > c_*$ and integrating $(\ref{pfas})$ over $(-\infty,t)$, we find that
$$
\phi'(t) = c\phi(t) +\int_t^{+\infty}(q\phi(s) - \int_{-h}^0\phi(s+c\sigma)d\mu_+(\sigma) +Q(s))ds =\lambda_1(c)e^{\lambda_1(c)t}+ V(t), 
$$ 
where $V(t)= O(e^{(\lambda_1(c)+\varepsilon)t}), \quad t \to -\infty$. A similar computation in the case $c=c_*$ ends the proof of Lemma \ref{EDE2}. 
\hfill $\square$
\end{proof}

\subsection{Proof of the semi-wavefront uniqueness.} Suppose that $\phi, \psi$ satisfy (\ref{pf}) and for some $\lambda >0$, set 
$y(t) = (\psi(t)-\phi(t))e^{-\lambda t}$. Then 
\begin{equation} \label{ast}
y''(t)-(c-2\lambda)y'(t)+(\lambda^2 - c\lambda-q)y(t)+ \int_{-h}^0y(t+cs)e^{\lambda c s}d\mu_+(s)+ R(t,\tilde y_t) =0, 
\end{equation}
where 
$$
R(t,\tilde y_t)=  e^{-\lambda t}\left[f(\tilde \psi_t)- f(\tilde \phi_t) -f'(0)(\tilde \psi_t- \tilde \phi_t) \right]. 
$$
Then condition {\bf (S)} implies that, for some  $T$ and $K>0$, 
$$|R(t,\tilde y_t)| \leq K |y(t+c\cdot)|_C(|\tilde \phi_t|_C^{\alpha}+|\tilde \psi_t|_C^{\alpha}), \quad t \leq T.$$
First, we consider the non-critical case when $c> c_*$ and $\lambda_1(c) < \lambda_2(c)$.  Choose some $\lambda \in (\lambda_1(c), \lambda_2(c))$.
By Lemma \ref{EDE2}, without loss of generality, we can assume that 
$$
(\psi(t), \psi'(t)) \ \mbox{and also}\   (\phi(t), \phi'(t)) = (1,\lambda_1(c)) e^{\lambda_1(c)t} + O(e^{(\lambda_1(c)+\epsilon)t}), \ t \to -\infty. 
$$
Clearly, this implies that $(y(t),y'(t)) = O(e^{(\lambda_1(c)-\lambda +\epsilon)t})$ and  $R(t,\tilde y_t) = O(e^{\lambda_1(c)\alpha t}|\tilde y_t|_C)$ as $t \to -\infty$. Then $y(-\infty)=0$ and 
arguing as in \cite[Proposition 6.1]{MP}, we conclude that 
$$
y(t) = A e^{(\lambda_2(c)-\lambda) t} + O(e^{(\lambda_2(c)-\lambda +\epsilon)t}), \ t \to -\infty. 
$$
By interchanging, if necessary, the roles of $\phi$ and $\psi$, we may assume that $A\geq 0$. 

Suppose first that  $A=0$, then $y(t) =  O(e^{(\lambda_2(c)-\lambda +\epsilon)t}), \ t \to -\infty$. Note that  the eigenvalues of the homogeneous part of equation (\ref{ast}) coincide with the zeros of $\chi(z,c)$ shifted by $-\lambda$. Therefore, since  $\chi(z,c)$ does not have eigenvalues with $\Re \lambda_j > \lambda_2(c)$,  by \cite[Proposition 6.1]{MP} we conclude that $y(t)$ is a small solution of  the asymptotically autonomous (at $-\infty$) linear equation (\ref{ast}).  Invoking now Lemma \ref{sms}, we conclude that $y(t)\equiv 0$. This means that $\phi(t) \equiv \psi(t)$ in the case when $A=0$. 
 
 Next, suppose that $A>0$. Then $y(t) >0$ on some maximal interval $\mathcal{I}= (-\infty, \theta)$, $\theta \in \R\cup \{+\infty\}$. Since $y(-\infty)=y(\theta)=0$, 
 we conclude that $y(t), \ t \in \mathcal{I},$ reaches its positive absolute maximum at some point $\zeta \in \mathcal{I}$, where $y(\zeta) > 0$ and 
 $$
 y''(\zeta)\leq 0, \ y'(\zeta)=0,  \ \int_{-h}^0y(\zeta+cs)e^{\lambda c s}d\mu_+(s) \leq \int_{-h}^0e^{\lambda c s}d\mu_+(s)y(\zeta), \ R(\zeta,\tilde y_\zeta) \leq 0. 
 $$
However, then (\ref{ast}) yields the following contradiction: 
$$
0 = y''(\zeta)-(c-2\lambda)y'(\zeta)+(\lambda^2 - c\lambda-q)y(\zeta)+ \int_{-h}^0y(\zeta+cs)e^{\lambda c s}d\mu_+(s)+ R(\zeta,\tilde y_\zeta) \leq
$$
$$
(\lambda^2 - c\lambda-q+  \int_{-h}^0e^{\lambda c s}d\mu_+(s))y(\zeta) = \chi(\lambda,c) y(\zeta) <0.   
$$
This proves the uniqueness of every non-critical semi-wavefront. 

Finally, we consider  critical case, $c= c_*$ (usually more difficult, cf. \cite{AGT, CC}). Then  $\lambda_1(c) = \lambda_2(c)$ and we take  $\lambda = \lambda_1(c)$.  We will need the following equivalent form of  the relations $\chi(\lambda,c) = \chi'(\lambda,c)=0$:
\begin{equation}\label{rel2}
\lambda^2 - c\lambda-q+  \int_{-h}^0e^{\lambda c s}d\mu_+(s)=0, \quad 2\lambda - c+  \int_{-h}^0cse^{\lambda c s}d\mu_+(s)=0.
\end{equation}
Again invoking  Lemma \ref{EDE2}, without loss of generality we can assume that 
$$
(\psi(t), \psi'(t)) \ \mbox{and also}\   (\phi(t), \phi'(t)) = -(1,\lambda_1(c)) e^{\lambda_1(c)t}(t+O(1)), \ t \to -\infty. 
$$
Clearly, this implies that $(y(t),y'(t)) = O(1)$ while  $R(t, \tilde y_t) = O(|\tilde y_t|_Ce^{\lambda_1(c)\alpha t})$ as $t \to -\infty$. By  \cite[Proposition 6.1]{MP}, we conclude that, for some small $\epsilon >0,$
$$
y(t) = B  + O(e^{\epsilon t}), \quad y'(t) =  O(e^{\epsilon t}),  \ t \to -\infty, \quad y(+\infty)=0. 
$$
By interchanging, if necessary, the roles of $\phi$ and $\psi$, we may assume that $B\geq 0$. If $B=0$, then the same argument as in the non-critical case with  $A=0$ shows that  $y(t)\equiv 0$ proving the uniqueness of the semi-wavefront profile. If $B>0$, then $y(t) >0$ on some maximal interval $\mathcal{I}= (-\infty, \theta)$, $\theta \in \R\cup \{+\infty\}$.  Using (\ref{rel2}), we can rewrite  equation (\ref{ast})  as follows: 
$$
0 = y''(t)-(c-2\lambda)y'(t)+(\lambda^2 - c\lambda-q)y(t)+ \int_{-h}^0y(t+cs)e^{\lambda c s}d\mu_+(s)+ R(t,\tilde y_t) =
$$
$$
y''(t)-(c-2\lambda)y'(t)+\int_{-h}^0[y(t+cs)-y(t)]e^{\lambda c s}d\mu_+(s)+ R(t,\tilde y_t) =
$$
$$
 y''(t)- \int_{-h}^0cse^{\lambda c s}d\mu_+(s) y'(t)+\int_{-h}^0\int_0^1cs y'(t+c\sigma s)d\sigma e^{\lambda c s}d\mu_+(s)+ R(t,\tilde y_t) =
$$
\begin{equation}\label{sed}
 y''(t)+\int_{-h}^0\int_0^1cs [y'(t+c\sigma s)-y'(t)]d\sigma e^{\lambda c s}d\mu_+(s)+ R(t,\tilde y_t) =0. 
\end{equation}
Thus 
\begin{equation}\label{seda}
y''(t)+\int_{-h}^0 (cs)^2 e^{\lambda c s} d\mu_+(s) \int_0^1\sigma d\sigma  \int_0^1y''(t+c\tau\sigma s)d\tau  + R(t,\tilde y_t) =0.
\end{equation}
Because of (\ref{sed}), similarly to $y'(t)$ and $R(t,\tilde y_t)$, the second derivative $y''(t)$ has exponential decay at $-\infty$ and $+\infty$. Therefore,  integrating (\ref{seda}) with respect to $t$ on $(-\infty,\theta]$ and using Fubini's theorem, we find that 
$$
0= y'(\theta)+\int_{-h}^0 (cs)^2 e^{\lambda c s} d\mu_+(s) \int_0^1\sigma d\sigma  \int_0^1y'(\theta+c\tau\sigma s)d\tau  + \int_{-\infty}^\theta R(t,\tilde y_t)dt= 
$$
$$
 y'(\theta)+\int_{-h}^0 cs e^{\lambda c s} d\mu_+(s) \int_0^1 [y(\theta+c\sigma s)-y(\theta)] d\sigma   + \int_{-\infty}^\theta R(t,\tilde y_t)dt= 
$$
$$
 y'(\theta)+\int_{-h}^0 cs e^{\lambda c s} d\mu_+(s) \int_0^1 y(\theta+c\sigma s) d\sigma   + \int_{-\infty}^\theta R(t,\tilde y_t)dt.
$$
Since all three terms in the latter line are non-negative, we conclude that 
$$
y'(\theta)=\int_{-h}^0 cs e^{\lambda c s} d\mu_+(s) \int_0^1 y(\theta+c\sigma s) d\sigma  = \int_{-\infty}^\theta R(t,\tilde y_t)dt=0. 
$$
In addition, we know that  continuous function $R$ satisfies $R(t,\tilde y_t)\leq 0$ on $(-\infty,\theta]$. 
This means that $R(t,\tilde y_t)= 0$ for all $t \in (-\infty,\theta]$ so that $y(t)$ is a bounded solution of the linear equation 
$$
y''(t)-(c-2\lambda)y'(t)+(\lambda^2 - c\lambda-q)y(t)+ \int_{-h}^0y(t+cs)e^{\lambda c s}d\mu_+(s)=0, \quad t \leq \theta.
$$
Thus $y(t)\equiv B$ on $(-\infty,\theta]$ and, in particular, $y(\theta)=B >0$. The obtained contradiction shows that actually $B=0$ that completes  the proof of Theorem \ref{main}.  \hfill $\square$

\section*{Appendix}
This section contains the proof of Lemma \ref{STA}.  

By  {\bf (ND)}   and (\ref{rel2}), equation $\chi(z,c) =0$ does  not have real roots when $c=0$. If $c>0$, then it can be rewritten in the following equivalent form 
\begin{equation}\label{ep}
z+q-\epsilon z^2 = \int_{-h}^0e^{zs}d\mu_+(s), \quad z:=\lambda c, \ \  \epsilon = c^{-2}.
\end{equation}
On the left [respectively, right] side of (\ref{ep}) we have a  strictly convex upward [respectively, downward]  function, so that equation (\ref{ep}) can have at most two real roots counting multiplicity. 
In fact, since $p>q$ and  positive function $\int_{-h}^0e^{zs}d\mu_+(s)$ is non-increasing in $z$,  we deduce the existence of $\epsilon_*>0$ such that (\ref{ep})  has exactly  
two simple real roots if and only if $\epsilon \in (0, \epsilon_*)$ and has a double real root if and only if $\epsilon=\epsilon_*$.  

The above analysis implies  all conclusions of 
Lemma \ref{STA} except for the last (and key) assertion concerning the dominance of the real zeros $\lambda_1(c) \leq \lambda_2(c)$. 
Actually, the vertical strip $\lambda_1(c) \leq \Re z \leq  \lambda_2(c)$ does not contain any complex zero $w=a+ib, \ b \not =0$ of $\chi(z,c)$ for otherwise, with $z_1 <0\leq z_2$ denoting the roots of $z^2-cz-q=0$,  we get the following contradiction: $|w^2-cw-q| =  $
$$
|w -z_1||w-z_2| > |a -z_1||a-z_2| =  ca+q -a^2 \geq   \int_{-h}^0e^{acs}d\mu_+(s) \geq |\int_{-h}^0e^{wcs}d\mu_+(s)|.
$$
Note that $z_1 < \lambda_1(c)\leq  \lambda_1(c) < z_2$.  

Now,  consider $c >c_*$ and choose $\nu$  such that $\lambda_1(c) < \nu < \lambda_2(c)$. The above argument shows that 
$|w^2-cw-q| >  |\int_{-h}^0e^{wcs}d\mu_+(s)|$ for all $w \in \Gamma$, where $\Gamma$ is the boundary of the rectangle 
$[\nu, \zeta] \times [-k,k] \in \R^2\simeq \C$ with $\zeta$ and $k$ sufficiently large. By Rouch\'e's theorem this implies that 
$\chi(z,c)$ and $z^2-cz-q$ have the same number of zeros with $\Re z > \nu$, i.e. exactly one zero.    Finally, if equation $\chi(z,c_*)=0$ has at least
one root $z_0$ with $\Re z_0 > \lambda_1(c_*)$, then  by Hurwitz's theorem from the complex analysis, we conclude that $\chi(z,c)$ also has at least
one root $\tilde z_0$ with $\Re \tilde z_0 > \lambda_1(c)$ for all $c >c_*$ close to $c_*$, a contradiction. This completes the proof of  Lemma \ref{STA}. \hfill $\square$

\section*{Acknowledgments}  \noindent   This work was supported by FONDECYT (Chile) through the Postdoctoral Fondecyt 2016 program with project number 3160473 (A. Solar) and FONDECYT  
project 1150480 (S. Trofimchuk).

\end{document}